\documentclass[1pt]{article}
\usepackage{amsmath, amssymb, amsthm, amsfonts, cases}
\usepackage{mathrsfs}
\usepackage{url}
\usepackage{authblk}
\usepackage[usenames]{color}
\usepackage{geometry}
\theoremstyle{plain}
\newtheorem{thm}{Theorem}[section]

\newtheorem{pro}[thm]{Problem}
\newtheorem{lem}[thm]{Lemma}

\theoremstyle{definition}

\newtheorem{ass}{Assumption}[section]
\newtheorem{rmk}{Remark}[section]

\newcommand{\eps}{\varepsilon}

\newcommand{\R}{\mathbb{R}}

\makeatletter\@addtoreset{equation}{section} \makeatother

\begin{document}

\title{Linear-Quadratic Optimal Control Problems for Mean-Field
 Backward Stochastic Differential Equations with Jumps
\thanks{This work was supported by the Natural Science Foundation of Zhejiang Province
for Distinguished Young Scholar  (No.LR15A010001),  and the National Natural
Science Foundation of China (No.11471079, 11301177) }}

\date{}

\author[a]{Maoning Tang}
\author[a]{Qingxin Meng\footnote{Corresponding author.
\authorcr
\indent E-mail address: mqx@zjhu.edu.cn(Q. Meng)}}

\affil[a]{\small{Department of Mathematical Sciences, Huzhou University, Zhejiang 313000, China}}

\maketitle

\begin{abstract}
\noindent
This paper  is concerned with  a linear
quadratic (LQ, for short) optimal control problem for mean-field backward stochastic
differential equations (MF-BSDE, for short)
driven by a Poisson random martingale measure and a Brownian motion. Firstly, by the classic  convex variation
principle, the existence and
uniqueness of the optimal control is
established.  Secondly, the optimal 
 control is characterized by the stochastic Hamilton system which turns out to be
a linear fully coupled mean-field forward-backward stochastic differential equation with jumps by the duality method. 
Thirdly, in terms of  a decoupling technique, the stochastic Hamilton system is decoupled by
introducing  two Riccati equations and 
a MF-BSDE with  jumps. Then an  explicit representation for the optimal control is obtained.

\end{abstract}


\section{Introduction}

Stochastic optimal control problems of mean-field type recently are extensively 
studied, due to their comprehensive  practical applications especially in
  economics and finance. 
  Different from  the classic stochastic optimal control problem, the   probability distribution of the state and the control are
  involved in  the coefficients of
   the state equation  and cost  functional
  which leads to    a time-inconsistent
  optimal control problem, so the dynamic programming principle (DPP) is not  effective  and  many  researchers try to solve this type of optimal control problems by discuss
  the stochastic maximum  principle (SMP) instead of trying extensions of DPP. One can refer  to  [1-3],[6],[9], [12],[16], [17], [22],[23],[25-28] for  more results on the stochastic maximum principles for different 
  types of mean-field stochastic models.

In 2013, the continuous-time mean-field LQ problem for  mean-field forward stochastic differential equation (MF-FSDE) were systematically studied  by Yong [29], where the optimal control is represented as a state feedback form by introducing two Riccati differential equations.  Since [29], many advances have been made on LQ control problem for stochastic system of
mean-field type  (cf., for example [8],  [10],
[13]-[15],[18-21],[26]). In 2016,
Tang and  Meng [24] extended  continuous-time mean-field LQ control
 problem to jump diffusion system  and established the corresponding theoretical results.   It is well-knew that  the adjoint equation of a controlled  MF-FSDE is a MF-BSDE. So it is not until Buckdahn,
Djehiche, Li  and Peng [4]; Buckdahn, Li  and Peng [5]
established the theory of the mean-field BSDEs that the stochastic
maximum principle for the optimal control system of mean-field type has become a popular topic.
  Since a BSDE is a well-defied dynamic system
itself and has important applications in mathematical finance, it is necessary and natural to consider the optimal control problem of BSDE. In 2016, Li, Sun and Xiong [11] first studied  the  stochastic LQ problems under continuous-time MF-BSDEs  and  established the corresponding  fundamental theoretical results.
 The purpose of this paper is to
 extend  continuous-time mean-field LQ 
 problem to backward jump diffusion system of mean-field type and establish the corresponding theoretical results.
 We first give notations used throughout  our 
 paper and formulate our LQ problem in section 2. In section 3, we prove  the existence and uniqueness of the optimal control by classic
convex variation principle under standard  assumptions. In section 4, we first establish the dual characterization of the optimal control by stochastic Hamiltonian system.
Here  the stochastic Hamiltonian system turns out to be a
full-coupled  forward-backward stochastic differential equation of mean-field type
with jump, which is very difficult to be solved. Then we introduce two Riccati equations and 
a MF-BSDE to decouple the stochastic Hamiltonian system.
At last, we present explicit formulas of
the optimal controls.

\section{Notations and Formulation of Problem}
\subsection{Notations}
~~~~Let $T$ be a fixed strictly positive real number and  $(\Omega,
\mathscr{F},\{\mathscr{F}_t\}_{0\leq t\leq T}, P)$ be a complete probability space on which a d-dimensional standard Brownian motion $\{W(t), 0\leq t\leq T\}$
is defined. Denote by $\mathscr{P}$
the $\mathscr{F}_t$-predictable $\sigma$-field on $[0, T]\times \Omega$ and by $\mathscr B(\Lambda)$
  the Borel $\sigma$-algebra of any topological space $\Lambda.$ Let $(E,\mathscr B (E), v
)$ be a measurable space with $v(E)<\infty $ and $\eta: \Omega\times D_\eta \longrightarrow E$ be an
$\mathscr F_t$-adapted stationary Poisson point process with characteristic measure $v$, where
 $D_\eta$  is a countable subset of $(0, \infty)$. Then the counting measure induced by $\eta$ is
$$
\mu((0,t]\times A):=\#\{s\in D_\eta; s\leq t, \eta(s)\in A\},~~~for~~~ t>0, A\in \mathscr B (E).
$$
And $\tilde{\mu}(de,dt):=\mu(de,dt)-v(de)dt$ is a compensated Poisson random martingale  measure which
is assumed to be independent of Brownian motion
$\{W(t), 0\leq t\leq T\}$. Assume $\{\mathscr{F}_t\}_{0\leq t\leq T}$ is $P$-completed natural
 filtration generated by $\{W(t),
0\leq t\leq T\}$ and $\{\iint_{A\times (0,t] }\tilde{\mu}(de,ds), 0\leq t\leq T, A\in \mathscr B (E) \}.$  In the following,  we introduce the basic
notations used throughout this paper.

$\bullet$~~$t$: $t\in[0, T)$.

$\bullet$~~$H$: a Hilbert space with norm $\|\cdot\|_H$.

$\bullet$~~$\langle\alpha,\beta\rangle:$ the inner product in
$\mathbb{R}^n, \forall \alpha,\beta\in\mathbb{R}^n.$

$\bullet$~~$|\alpha|=\sqrt{\langle\alpha,\alpha\rangle}:$ the norm
of $\mathbb{R}^n,\forall \alpha\in\mathbb{R}^n.$

$\bullet$~~$\langle A,B\rangle=tr(AB^\top):$ the inner product in
$\mathbb{R}^{n\times m},\forall A,B\in \mathbb{R}^{n\times m}.$
Here denote by $B^\top$ the transpose of a matrix B.

$\bullet$~~$|A|=\sqrt{tr(AA^\top)}:$ the norm of $A$.

$\bullet$~~$S^n:$ the set of all $n\times n$ symmetric matrices.

$\bullet$~~ $S^n_+:$ the subset of all non-negative definite matrices of $S^n.$

$\bullet$~~$S_{\mathscr{F}}^2(t,T;H):$ the space of all $H$-valued
and ${\mathscr{F}}_t$-adapted  c\`{a}dl\`{a}g processes
$f=\{f(t,\omega),\ (t,\omega)\in[t,T]\times\Omega\}$ satisfying
$$\|f\|_{S_{\mathscr{F}}^2(t,T;H)}^2\triangleq{
\mathbb E\bigg[\displaystyle\sup_{t\leq s \leq T}\|f(s)\|_H^2ds}\bigg]<+\infty.$$

$\bullet$~~$ M_{\mathscr{F}}^2(t,T;H):$ the space of all $H$-valued
and ${\mathscr{F}}_s$-adapted processes $f=\{f(s,\omega),\
(s,\omega)\in[0,T]\times\Omega\}$ satisfying
$$\|f\|_{M_{\mathscr{F}}^2(t,T;H)}^2\triangleq {\mathbb E\displaystyle\bigg[\int_t^T\|f(s)\|_H^2dt}\bigg]<\infty.$$

$\bullet$~~${M}^{\nu,2}( E; H):$ the space of all H-valued measurable
  functions $r=\{r(\theta), \theta \in E\}$ defined on the measure space $(E, \mathscr B(E); v)$ satisfying
$$\|r\|_{{ M}^{\nu,2}( E; H)}^2\triangleq{\displaystyle\int_E\|r(\theta)\|_H^2v(d\theta)}<~\infty.$$

 $\bullet$~~ ${M}_{\mathscr{F}}^{\nu,2}{([t,T]\times  E; H)}:$ the  space of all ${M}^{\nu,2}( E; H)$-valued
and ${\cal F}_s$-predictable processes $r=\{r(s,\omega,e),\
(s,\omega,e)\in[t,T]\times\Omega\times E\}$ satisfying
$$\|r\|_{M_{\mathscr F}^{\nu,2}([t,T]\times  E; H)}^2\triangleq {\mathbb E\bigg[\int_t^T\displaystyle\|r(s,\cdot)\|^2_
{{M}^{\nu,2}( Z; H)}dt}\bigg]<~\infty.$$
$\bullet$~~$L^2(\Omega,{\mathscr {F}},P;H):$ the space of all
$H$-valued random variables $\xi$ on $(\Omega,{\mathscr  {F}},P)$
satisfying
$$\|\xi\|_{L^2(\Omega,{\mathscr{F}},P;H)}\triangleq \mathbb E[\|\xi\|_H^2]<\infty.$$

\subsection{ Formulation of Problem}

Consider the following controlled linear MF-BSDE  driven by
 Brownian motion $\{W(s)\}_{t\leq s\leq T}$ and Poisson random  martingale  measure $\{\tilde{\mu}(d\theta,ds)\}_{t\leq s\leq T}$
\begin{equation}\label{eq:1.1}
\left\{\begin {array}{ll}
  dY(s)=&\bigg\{A(s)Y(s)+\bar A(s)\mathbb E [Y(s)]
  +B(s)u(s)+\bar B(s)\mathbb E [u(s)]+C(s)Z(s)+\bar C(s)\mathbb E [Z(s)]
  \\&+\displaystyle\int_{E} D(s,e)R(s,e)\nu (de)
  +\displaystyle\int_{E}\bar D(s,e) \mathbb E [R(s,e)]\nu (de)\bigg\}dt+ Z(s)dW(s)
  \\&+\displaystyle
  \int_{E} R(s,e)\tilde{\mu}(d\theta,
  ds), s\in [t, T],
   \\Y(T)=&\xi,
\end {array}
\right.
\end{equation}
with the following quadratic cost {functional}
\begin{eqnarray}\label{eq:1.2}
\begin{split}
 J(t,\xi; u(\cdot))=&\mathbb E\bigg\{\int_t^T\bigg(\langle Q(s)Y(s),
Y(s)\rangle+ \langle \bar {
Q}(s)\mathbb E[Y(s)], \mathbb E[Y(s)]\rangle
+\langle N_1(s)Z(s), Z(s)\rangle+\langle
\bar N_1(s)\mathbb E[Z(s)], \mathbb E[Z(s)]\rangle
\\&+ \int _{E}\langle N_2(s,e )R(s,e), R(s,e)\rangle\nu(de)+ \int _{E}\langle \bar N_2(s,e )\mathbb E[R(s,e)], \mathbb E[R(s,e)]\rangle\nu(de)
\\&
+\langle \bar N_3(s)u(s), u(s)\rangle+\langle \bar
{N}_3(s)\mathbb E[u(s)], \mathbb E[u(s)]\rangle
\bigg)ds\bigg]
\\&+\langle GY(t), Y(t)\rangle]
 +\langle \bar{ G}\mathbb E[Y(t)], \mathbb
E[Y(t)]\rangle \bigg\},
\end{split}
\end{eqnarray}
where $A(\cdot), \bar A(\cdot), B(\cdot), \bar B(\cdot), C(\cdot), \bar C(\cdot), D(\cdot, \cdot), \bar D(\cdot,\cdot)
, Q(\cdot), \bar Q(\cdot),
N_1(\cdot),\bar N_1(\cdot), N_2(\cdot,\cdot), \bar N_2(\cdot,\cdot),
 N_3(\cdot), \bar N_3(\cdot) $  are given
deterministic matrix-valued functions;
 $G$ and $\bar G$ are given matrices; $\xi$ is
 an $\mathscr F_T-$measurable random variable.

In the above, $u(\cdot)$ is our admissible control process. In this paper, a predictable stochastic process  $u(\cdot)$ is said to be an admissible control, if $ u(\cdot)\in
M_{\mathscr F}^2(t, T;\mathbb R^m)$. The set of all admissible controls is denoted by ${\cal A}[t,T]$. It is obviously that    ${\cal A}[t,T]$  is
a reflexive Banach space whose norm denoted by $||\cdot||_{{\cal A }[t,T]}$  is defined by
$$||u(\cdot)||_{{\cal A}[t,T]}= \bigg\{{\mathbb E\displaystyle\bigg[\int_0^T|u(t)
|^2dt}\bigg]\bigg\}^{\frac{1}{2}}, \forall u(\cdot)\in {\cal A}[t,T]. $$
For any admissible control $u(\cdot),$ the strong solution of the system { \eqref{eq:1.1}},  denoted by $(Y^{(t,\xi,u)}(\cdot),Z^{(t,\xi,u)}(\cdot), R^{(t,\xi, u)}(\cdot))$ or $(Y(\cdot), Z(\cdot), R(\cdot,\cdot))$
  if its dependence on
  admissible control $u(\cdot)$ is clear from  the context,   is called the
state process corresponding to the control process $u(\cdot)$, and
 $(u(\cdot); Y(\cdot), Z(\cdot), R(\cdot))$ is called an
admissible pair.

Our mean-field backward stochastic
linear quadratic (LQ) optimal  control problem
 with jump can be stated
as follows:
\begin{pro}\label{pro:1.1}
For any given $\xi\in L^2(\Omega,{\mathscr {F}},P;\mathbb R^n),$
find an admissible control ${u}^*(\cdot)\in {\cal A}[t, T]$ such that
\begin{equation}\label{eq:b7}
J(t,\xi;{u}^*(\cdot))=\displaystyle\inf_{u(\cdot)\in {{\cal A}[t, T]}}J(t, \xi, u(\cdot)).
\end{equation}
\end{pro}
Any  ${u}^*(\cdot)\in {{\cal A}[t, T]}$ satisfying  the above is called an
optimal control process of Problem \ref{pro:1.1} and the corresponding state process $( Y^*(\cdot), Z^*(\cdot), R^*(\cdot, \cdot))$ is
called the optimal state process. Correspondingly $(u^*(\cdot);
Y^*(\cdot), Z^*(\cdot), R^*(\cdot, \cdot))$ is called an optimal pair of
Problem \ref{pro:1.1}.

Throughout this paper, we make the following assumptions on the coefficients.

\begin{ass}\label{ass:1.1}
 The matrix-valued functions $A, \bar A, C, \bar C,
 Q, \bar Q, N_1,\bar N_1:[0, T]\rightarrow \mathbb R^{n\times n};
B,\bar B, :[0, T]\rightarrow \mathbb R^{n\times m}; D, \bar D:[0, T]\rightarrow {\cal L}^{v,2}(E; \mathbb R^{ n\times n}), N_2, \bar N_2:[0,T]\rightarrow  {\cal L} ^{v,2}(E; \mathbb R^{n\times n}); N_3, \bar N_3:[0, T]\rightarrow \mathbb R^{m\times m}$ are uniformly bounded measurable functions.
\end{ass}

\begin{ass}\label{ass:1.2}
 The matrix-valued functions $Q, Q+\bar Q, N_1, N_1+\bar N_1,N_2+\bar N_2, N_3, N_3+\bar N_3 $ are a.e. nonnegative
matrices, and $G, G+\bar G$ are nonnegative matrices.  Moreover, $N_3,N_3+\bar N_3 $  uniformly positive, i.e. for $\forall u\in \mathbb R^m$ and a.s. $s\in [t, T]$,
$ \langle N_3(s)u, u \rangle \geq \delta \langle u, u\rangle
$ and $ \langle (N_3(s)+\bar N_3(s))u, u \rangle \geq \delta \langle  u, u\rangle,
$  for some positive constant
$\delta.$
\end{ass}

The following result gives the well-posedness of the state equation as well as some useful estimates.

\begin{lem}\label{lem:1.1}
Let Assumption \ref{ass:1.1}  be satisfied. Then for any  $(\xi, u(\cdot))\in L^2(\Omega,{\mathscr {F}_T},P;\mathbb R^n)\times {\cal A} [t,T]$,
the state equation
\eqref{eq:1.1}
has a unique solution $ \Lambda(\cdot):=( Y(\cdot),
Z(\cdot),  R(\cdot, \cdot))\in M^2[t,T]:=S_{\mathscr F}^2 ( t, T; \mathbb R^n) \times  M_{\mathscr F}^2 ( t, T; \mathbb R^n) \times {M}_{\mathscr{F}}^{\nu,2}{([t,T]\times  E; \mathbb R^n)} .$
Moreover, we have the following estimate
\begin{eqnarray}\label{eq:1.4}
\begin{split}
  ||\Lambda(\cdot)||^2_{M^2[t, T]}\leq K 
  \bigg\{||u(\cdot)||_{{\cal A}[t, T]}^2+E\big[|\xi|^2\big]\bigg\}
  \end{split}
\end{eqnarray}
and
\begin{eqnarray}\label{eq:1.5}
\begin{split}
  |J(t,\xi;  u(\cdot))|< \infty.
  \end{split}
\end{eqnarray}
Suppose that $ \bar \Lambda(\cdot):=(\bar Y(\cdot),
\bar Z(\cdot), \bar R(\cdot,\cdot))$  is the solution to  the state equation \eqref{eq:1.1} corresponding to another  $(\bar\xi, \bar u(\cdot))\in L^2(\Omega,{\mathscr {F}}_T,P;\mathbb R^n)\times {\cal A} [t,T],$  then we have the
following estimate
\begin{eqnarray} \label{eq:1.6}
\begin{split}
   ||\Lambda(\cdot)-\bar\Lambda (\cdot)||^2_{M^2[t, T]}
\leq K \bigg \{
  ||u(\cdot)-\bar u(\cdot)||^2_{{\cal A}[t, T]}+\mathbb E\big[
  |\xi-\bar \xi|^2\big]\bigg\}.
  \end{split}
\end{eqnarray} 
Here we define
$$ ||\Lambda(\cdot)||^2_{M^2[t, T]}=:\mathbb E \bigg[\sup_{t\leq s\leq T}| Y(s)|^2\bigg]
  +\mathbb E \bigg[\int_t^ T| Z(s)|^2ds\bigg]
  +\mathbb E \bigg[\int_t^ T\int_E| R(s,e)|^2 \nu(de)ds\bigg].$$

\end{lem}

\begin{proof}
  The existence and uniqueness of the  solution can be  obtained by a standard argument using the contraction mapping theorem.
  For the estimates \eqref{eq:1.4} and \eqref{eq:1.6}, we can
  easily  obtain them  by  applying the It\^{o} formula to $|Y(\cdot)|^2$ and
   $|Y(\cdot)-\bar Y(\cdot)|^2 $, Gronwall inequality and B-D-G inequality.
   For the  estimate \eqref{eq:1.5}, using  Assumption
   \ref{ass:1.1} and  the  estimate  \eqref{eq:1.4},   we have

   \begin{eqnarray}\label{eq:1.7}
     \begin{split}
       |J(t, \xi; u(\cdot))|\leq  K \bigg\{
       ||\Lambda(\cdot)||_{M^2[t,T]}^2
       +||u(\cdot)||_{{\cal A}[t,T]}^2\bigg\}
  \leq  K \bigg\{||u(\cdot)||_{{\cal A}[t,T]}^2+ \mathbb E[|\xi|^2]\bigg\}
   < \infty,
     \end{split}
   \end{eqnarray}
   where we have used the elementary inequality:
   for any $\Phi \in {L^2(\Omega,{\mathscr {F}},P;H)},$
   \begin{eqnarray}
     \begin{split}
       ||\mathbb E [\Phi]||^2_H\leq  \mathbb E||\Phi||_H^2.
     \end{split}
   \end{eqnarray}
The proof is complete.
\end{proof}
Therefore, by Lemma \ref{lem:1.1}, we know that
Problem \ref{pro:1.1}
is well-defined.

\section{ Existence and Uniqueness of Optimal Control}

In this section, we study the existence and uniqueness of
the optimal control of Problem \ref{pro:1.1}.  To this end,  we first
establish
some elementary properties of the
cost functional.
\begin{lem}\label{lem:6.2}
  Let Assumptions  \ref{ass:1.1} and \ref{ass:1.2}
  be satisfied. Then  for any $(t,\xi)\in [0,T)\times L^2(\Omega,{\mathscr {F}}_T,P;\mathbb R^n),$ the cost functional $J(t,\xi;u(\cdot))$
is continuous over ${\cal A}[t,T].$
\end{lem}

\begin{proof}
 Suppose that $(u (\cdot);  \Lambda(\cdot))=(u (\cdot);  Y(\cdot),
Z(\cdot),  R(\cdot,\cdot))$ and $( \bar u (\cdot);  \bar \Lambda(\cdot))=(\bar u (\cdot); \bar Y(\cdot), \bar Z(\cdot), \bar R(\cdot,\cdot))$ be any two admissible control pairs.
 Under Assumptions \ref{ass:1.1} and \ref{ass:1.2}, from  the definition of the cost functional $J(t,\xi; u(\cdot))$ (see \eqref{eq:1.2}),
 we  have
\begin{eqnarray}\label{eq:5.10}
&&| J (t,\xi; u (\cdot)) - J (t,\xi; \bar u(\cdot) ) |^2
\\&\leq& K \bigg \{   ||\Lambda(\cdot)-\bar\Lambda (\cdot)||^2_{M^2[t, T]}+||u(\cdot)-\bar u(\cdot)||^2_{{\cal A}[t, T]}\bigg \}  \times \bigg\{ ||\Lambda(\cdot)||^2_{M^2[t, T]}+||u(\cdot)||^2_{{\cal A}[t, T]}+
||\bar\Lambda (\cdot)||^2_{M^2[t, T]}+||\bar u(\cdot)||^2_{{\cal A}[t, T]}\bigg \}
 \ . \nonumber
\end{eqnarray}
 Using the estimates \eqref{eq:1.4} and \eqref{eq:1.6} lead to
\begin{eqnarray}
| J (t,\xi; u (\cdot)) - J (t,\xi; \bar u (\cdot)) |^2
&\leq& K  \bigg \{   ||u(\cdot)-\bar u(\cdot)||^2_{{\cal A}[t, T]}\bigg \}  \times \bigg\{||u(\cdot)||^2_{{\cal A}[t, T]}+||\bar u(\cdot)||^2_{{\cal A}[t, T]}+\mathbb E[|\xi|^2]\bigg \} \ .
\end{eqnarray}
Thus, we get that
\begin{eqnarray}
J (t,\xi; u (\cdot)) - J (t,\xi; \bar u (\cdot)) \rightarrow 0 \ , \quad  as \quad u (\cdot) \rightarrow   \bar  u  (\cdot)
\quad in \quad {{\cal A}[t, T]} \ .
\end{eqnarray}
The proof is complete.

\end{proof}

 \begin{lem}\label{lem:6.3}
   Let
Assumptions \ref{ass:1.1} and \ref{ass:1.2}
  be satisfied. Then  for any given $(t,\xi)\in [0,T)\ \times L^2(\Omega,{\mathscr {F}}_T,P;\mathbb R^n),$
   the cost functional $J(t, \xi; u(\cdot))$ is
strictly convex   ${\cal A}[0,T].$ Moreover, the cost
functional $J(t,
\xi; u(\cdot))$ is coercive over ${\cal A}[0,T],$
i.e.,
$$\displaystyle\lim_ {\|u(\cdot)\|_{{\cal A}[t,T]}{\rightarrow
\infty}}J(t,\xi; u(\cdot))=\infty.$$
\end{lem}
\begin{proof}
Since   the weighting matrices in the cost
functional is not random, it is  easy  to  check that
  \begin{eqnarray}\label{eq:2.4}
\begin{split}
 J(t,\xi; u(\cdot))=& \displaystyle \mathbb E\bigg[\int_t^T\bigg(\langle Q(s)(Y(s)-\mathbb E[Y(s)]),
Y(s)-\mathbb E[Y(s)])\rangle + \langle  (Q+\bar {
Q})(s)\mathbb E[Y(s)], \mathbb E[Y(s)]\rangle
 \\&+\langle N_1(s)(Z(s)-\mathbb E[Z(s)]), Z(s)-\mathbb E [Z(s)]\rangle
 +\langle  (N_1(s)+\bar
{N}_1(s))\mathbb E[Z(s)], \mathbb E[Z(s)]\rangle
 \\&+
 \int_{E}\langle N_2(s,e )(R(s,e)-\mathbb E[R(s,e)]), R(s,e)-\mathbb E [R(s,e)]\rangle
 \nu(de)
 \\&+ \int_{E}\langle  (N_2(s,e)+\bar
{N}_2(s,e))\mathbb E[R(s,e)], \mathbb E[R(s,e)]\rangle \nu(de)
 \\&+\langle N_3(s)(u(s)-\mathbb E[u(s)]), u(s)-\mathbb E [u(s)]\rangle
 +\langle  (N_3(s)+\bar
{N}_3(s))\mathbb E[u(s)], \mathbb E[u(s)]\rangle\bigg)ds\bigg]
\\&+\mathbb E\bigg[\langle
G(Y(t)-\mathbb E [Y(t)], Y(t)-\mathbb E [Y(t)]\rangle +\langle  (G+\bar{ G})\mathbb E[Y(t)], \mathbb
E[Y(t)]\rangle \bigg].
\end{split}
\end{eqnarray}
Thus  the cost functional $J(t, \xi; u(\cdot))$ over
${\cal A}[t,T]$ is convex from the nonnegativity of the
 $N_1, N_1+\bar N_1, N_2, N_2+N_2, N_3+N_3, Q, Q+\bar Q, G, G+\bar G $. Actually, since
 $N_3, N_3+\bar N_3$ is uniformly positive, $J(t, \xi; u(\cdot))$ is strictly
convex. On the other hand, by Assumption \ref{ass:1.2}
and \eqref{eq:2.4} , we get
\begin{eqnarray}\label{eq:2.5}
  \begin{split}
    J(t,\xi; u(\cdot)) \geq& \mathbb E\bigg[\int_t^T
    \bigg (\langle N_3(s)(u(s)-\mathbb E[u(s)]), u(s)-\mathbb E[u(s)]\rangle
 +\langle  (N_3(s)+
{\bar N}_3(s))\mathbb E[u(s)], \mathbb E[u(s)]\rangle \bigg )ds\bigg]
\\ \geq &\delta \mathbb E\bigg[\int_t^T\langle u(s)-\mathbb E[u(s)], u(s)-\mathbb E[u(s)]\rangle ds\bigg]+
 \delta\mathbb E\bigg[\int_t^T \langle \mathbb E[u(s)], \mathbb E[u(s)]\rangle ds\bigg]
 \\=& \delta\mathbb E \bigg[\int_t^T |u(s)|^2ds\bigg]
\\=&\delta ||u(\cdot)||^2_{{\cal A}[t,T]}.
  \end{split}
\end{eqnarray}
Thus $\displaystyle\lim_ {\|u(\cdot)\|_{{\cal A}[t, T]}{\rightarrow \infty}}J(t,\xi; u(\cdot))=\infty.$ The proof is complete.
\end{proof}

 \begin{lem}\label{lem:2.3}
Let
Assumptions \ref{ass:1.1} and \ref{ass:1.2}
be satisfied.   Then  for any given $(t,\xi)\in [0,T)\ \times L^2(\Omega,{\mathscr {F}}_T,P;\mathbb R^n),$  the cost functional  $J(t,\xi; u(\cdot))$ is
Fr\`{e}chet differentiable over ${\cal A}[t,T]$ and
 the corresponding  Fr\`{e}chet
derivative $J'(t,\xi; u(\cdot))$ is  given by
\begin{eqnarray}\label{eq:2.6}
\begin{split}
 \langle J'(t,\xi; u(\cdot)),  v(\cdot) \rangle=&2\mathbb
E\bigg[\int_t^T\bigg(\langle Q(s)Y^{(t,\xi,u)}(s),  Y^{(t,0,v)}(s)\rangle
 +\langle \bar Q(s)\mathbb E [Y^{(x,u)}(s)],  \mathbb E[Y^{(t,0,v)}(s)]\rangle
 \\&+\langle N_1(s)Z^{(t,\xi,u)}(s),  Z^{(t,0,v)}(s)\rangle
 +\langle \bar N_1(s)\mathbb E [Z^{(t,\xi,u)}(s)],  \mathbb E[Z^{(t,0,v)}(s)]\rangle
  \\&+\int_{E}\langle N_2(s,e)R^{(t,\xi,u)}(s,e),  R^{(t,0,v)}(s,e)\rangle \nu(de)
\\&+\int_{E}\langle \bar N_2(s)\mathbb E [R^{(t,\xi,u)}(s,e)],  \mathbb E[R^{(t,0,\xi)}(s)]\rangle \nu(de)
\\&+\langle N_3(s)u(s), v(s)\rangle+\langle \bar N_3(s)\mathbb E[u(s)], \mathbb E [v(s)]\rangle \bigg)ds\bigg]
\\&+2\mathbb E\bigg[\langle G Y^{(t,\xi,u)}(t),  Y^{(t,0,v)}(t)\rangle+\langle
\bar G\mathbb E[Y^{(t,\xi,u)}(t)],  \mathbb E[Y^{(t,0,v)}(t)]\rangle\bigg] , ~~~~\forall
u(\cdot), v(\cdot)\in{\cal A}[t, T],
\end{split}
\end{eqnarray}
where $(Y^{(t,0,v)}(\cdot), Z^{(t,0,v)}(\cdot),R^{(t,0,v)}(\cdot, \cdot)) $ is the solution of the following
MF-BSDE

\begin{equation}\label{eq:2.7}
\left\{\begin {array}{ll}
  dY(s)=&\bigg\{A(s)Y(s)+\bar A(s)\mathbb E [Y(s)]
  +B(s)v(s)+\bar B(s)\mathbb E [v(s)]+C(s)Z(s)+\bar C(s)\mathbb E [Z(s)]
  \\&+\displaystyle\int_{E} D(s,e)R(s,e)\nu (de)
  +\displaystyle\int_{E}\bar D(s,e) \mathbb E [R(s,e)]\nu (de)\bigg\}ds+ Z(s)dW(s)
  \\&+\displaystyle
  \int_{E} R(s,e)\tilde{\mu}(d\theta,
  ds), s\in [t, T],
   \\Y(T)=& 0.
\end {array}
\right.
\end{equation}

\end{lem}

\begin{proof}

Let $u(\cdot)$ and $v(\cdot)$  be
two any given admissible control. For simplicity, the right hand side of \eqref {eq:2.6}
is denoted by  $\Delta^{u,v}.$
Since the state equation \eqref{eq:1.1} is linear, by the uniqueness of the solution of
the MF-BSDE,
it is easily to check that
\begin{eqnarray}\label{eq:2.8}
\begin{split}
  Y^{(t, \xi, u+v)}(s)=&Y^{(s,
  \xi, u)}(s)+Y^{(t,0, v)}(s),
  \\Z^{(t, \xi, u+v)}(s)=&Z^{(t,
  \xi, u)}(s)+Z^{(t,0, v)}(s), 
  \\R^{(t, \xi, u+v)}(s)=&R^{(t,
  \xi, u)}(s)+R^{(t,0, v)}(s), t\leq s\leq T.
  \end{split}
\end{eqnarray}
 Therefore, in terms of  \eqref{eq:2.8}
 and  the definition  of the
cost functional $J(x, u(\cdot))$ (see \eqref{eq:1.2}), it is easy to check
that
\begin{eqnarray} \label{eq:2.7}
  \begin{split}
 J(t,\xi; u(\cdot) +v(\cdot) )-J(t,\xi; u(\cdot) )= J(t,0;v(\cdot))+\Delta^{u,v}.
  \end{split}
\end{eqnarray}
 On the other hand,  the
  estimate \eqref{eq:1.7} leads to
\begin{eqnarray}
  \begin{split}
    |J(t,0;v(\cdot))| \leq & K||v(\cdot) ||^2_{{\cal A}[t, T]}.
  \end{split}
\end{eqnarray}
Therefore,
\begin{eqnarray}
  \begin{split}
  \displaystyle\lim_ {\|v(\cdot) \|_{{\cal A}[t, T]} {\rightarrow
0}}\frac{|J(t,\xi;  u(\cdot) +v(\cdot) )
-J(t,\xi; u(\cdot) )-\Delta^{u,v}|}{||v(\cdot) ||_{{\cal A}[t, T]}}=\displaystyle\lim_ {\|v(\cdot) \|_{{\cal A}[t, T]} {\rightarrow
0}}\frac{|J(t,0; v(\cdot)|}{||v(\cdot) ||_{{\cal A}[t, T]}}
=0,
  \end{split}
\end{eqnarray}
which gives  that $J(t, \xi; u(\cdot))$ has Fr\'{e}chet derivative  $\Delta^{u,v}.$
The proof is complete.
\end{proof}

\begin{rmk}
  Since the cost functional $J(t,\xi; u(\cdot))$
  is Fr\'{e}chet differentiable,  then
  it is also G\^{a}teaux  differentiable.
  Moreover,the G\^{a}teaux  derivative
  is the  Fr\'{e}chet derivative
  $\langle J'(t, \xi; u(\cdot)),  v(\cdot) \rangle.$
  In fact,  from \eqref{eq:2.7}, we have
  \begin{eqnarray} \label{eq:2.10}
  \begin{split}
  &\displaystyle\lim_ {\eps {\rightarrow
0}}\frac{J(t,\xi; u(\cdot) +\eps v(\cdot) )
-J(t,\xi;  u(\cdot) )}{\eps}
\\=&
\displaystyle\lim_ {\eps {\rightarrow
0}}\frac{J(t,0; \eps v(\cdot))+\Delta^{u,\eps v}}{\eps}
\\=&
\displaystyle\lim_ {\eps {\rightarrow
0}}\frac{\eps^2J(t,0; v(\cdot))+\eps\Delta^{u, v}}{\eps}
\\=&\Delta^{u, v}
\\=&
\langle J'(t,\xi; u(\cdot)),  v(\cdot) \rangle.
  \end{split}
\end{eqnarray}
\end{rmk}

Now  by  Lemma \ref{lem:6.2}-\ref{lem:2.3}, we can obtain the existence and uniqueness of optimal control.  This
result   is stated  as follows.
\begin{thm}\label{them:b1}
Let Assumptions \ref{ass:1.1} and \ref{ass:1.2}
be satisfied. Then Problem \ref{pro:1.1} has a unique
optimal control. \end{thm}

\begin{proof}

Since the admissible controls set ${\cal A}[t,T]=
 M^2_{\mathscr F}(t,T;\mathbb R^m)$ is a reflexive Banach space, in terms of  Lemma \ref{lem:6.2}-\ref{lem:2.3}, the uniqueness and existence of the optimal control of Problem \ref{pro:1.1}
 can be directly got from Proposition 2.12 of [7]
(i.e., the coercive, strictly convex and lower-semi continuous functional defined
on the reflexive Banach space has a unique
minimum value point). 
 The proof is complete.
\end{proof}

\begin{thm}\label{thm:2.5}
 Let Assumptions \ref{ass:1.1} and \ref{ass:1.2}
be satisfied. Then  a necessary and sufficient condition for an admissible control
 $u(\cdot)\in {\cal A}[t, T]$
 to be an optimal control  of Problem \ref{pro:1.1}  is that
  for any admissible control $v(\cdot) \in {\cal A}[t, T],$
  \begin{equation}\label{eq:b16}
    \langle  J'(t, \xi; u(\cdot) ), v(\cdot) \rangle = 0,
  \end{equation}
  i.e.
  \begin{eqnarray}\label{eq:2.13}
\begin{split}
 0=&2\mathbb
E\bigg[\int_t^T\bigg(\langle Q(s)Y^{(t,\xi,u)}(s),  Y^{(t,0,v)}(s)\rangle
 +\langle \bar Q(s)\mathbb E [Y^{(x,u)}(s)],  \mathbb E[Y^{(t,0,v)}(s)]\rangle
 \\&+\langle N_1(s)Z^{(t,\xi,u)}(s),  Z^{(t,0,v)}(s)\rangle
 +\langle \bar N_1(s)\mathbb E [Z^{(t,\xi,u)}(s)],  \mathbb E[Z^{(t,0,v)}(s)]\rangle
  \\&+\int_{E}\langle N_2(s,e)R^{(t,\xi,u)}(s,e),  R^{(t,0,v)}(s,e)\rangle \nu(de)
\\&+\int_{E}\langle \bar N_2(s)\mathbb E [R^{(t,\xi,u)}(s,e)],  \mathbb E[R^{(t,0,\xi)}(s)]\rangle \nu(de)
\\&+\langle N_3(s)u(s), v(s)\rangle+\langle \bar N_3(s)\mathbb E[u(s)], \mathbb E [v(s)]\rangle \bigg)ds\bigg]
\\&+2\mathbb E\bigg[\langle G Y^{(t,\xi,u)}(t),  Y^{(t,0,v)}(t)\rangle+\langle
\bar G\mathbb E[Y^{(t,\xi,u)}(t)],  \mathbb E[Y^{(t,0,v)}(t)]\rangle\bigg] , ~~~~\forall
u(\cdot), v(\cdot)\in{\cal A}[t, T].
\end{split}
\end{eqnarray}
\end{thm}

\begin{proof}
  For the necessary part,  suppose that $u(\cdot)$
  is  an optimal control.
   Then  from \eqref{eq:2.10},  for any admissible control $v(\cdot)$ and $0< \eps \leq 1,$
   we have
  \begin{eqnarray}
  \begin{split}
 \langle J'(t,\xi; u(\cdot)),  v(\cdot) \rangle= \displaystyle\lim_ {\eps {\rightarrow
0^{+}}}\frac{J(t,\xi; u(\cdot) +\eps v(\cdot) )-J(t,\xi; u(\cdot) )}{\eps}
\geq 0,
  \end{split}
\end{eqnarray}
and
\begin{eqnarray}
  \begin{split}
  -\langle J'(t,\xi; u(\cdot)),  v(\cdot)\rangle=\langle J'(t,\xi; u(\cdot)),  -v(\cdot) \rangle= \displaystyle\lim_ {\eps {\rightarrow
0^{+}}}\frac{J(t,\xi; u(\cdot) +\eps (-v(\cdot)) )-J(t,\xi; u(\cdot) )}{\eps}
\geq 0,
  \end{split}
\end{eqnarray}
which imply  that
 \begin{equation}\label{eq:b16}
    \langle  J'(t, \xi; u(\cdot) ), v(\cdot) \rangle = 0.
  \end{equation}
For the sufficient part, let $u(\cdot)$ be an given
admissible control, and suppose that
for any admissible control $v(\cdot),$
$ \langle  J'(t,\xi; u(\cdot) ), v(\cdot) \rangle = 0.$
Since the cost functional $J$
is convex, then we have
\begin{equation}\label{eq:b16}
 J(t,\xi; v(\cdot))-J(t,\xi; u(\cdot))  \geq  \langle  J'(t,\xi; u(\cdot) ), v(\cdot) \rangle = 0,
  \end{equation}
which
implies that
$u(\cdot)$ is an optimal
control. The proof is complete.
\end{proof}

\section { Stochastic Hamilton Systems, decoupling, and Riccati equations, Representations of optimal controls }
\subsection{Stochastic Hamilton Systems}

In this subsection, we give the characterization of optimal control of Problem \ref{pro:1.1} by 
the stochastic Hamilton system. To simplify our notation, in what follows, we shall often suppress the time variable s if no confusion
can arise.
\begin{thm}\label{thm:b2}
Let  Assumptions \ref{ass:1.1} and
\ref{ass:1.2}  be satisfied.
 Then, a necessary and
sufficient condition for an admissible pair $(u(\cdot); Y(\cdot), Z(\cdot), R(\cdot,\cdot))$ to be an optimal pair of  Problem \ref{pro:1.1} is that  the admissible
control $u(\cdot)$ satisfies
\begin{eqnarray} \label{eq:3.1000}
  \begin{split}
 &N_3(s)u(s)+\bar N_3(s)\mathbb E[u(s)]+B^\top(s)k({s-})+\bar B^\top(s) \mathbb E [k({s-})]=0, \quad a.e. a.s., s\in [t, T],
     \end{split}
\end{eqnarray}
where $k(\cdot) $ is the unique
solution of the following  MF-SDE
\begin {equation}\label{eq:3.2}
\left\{\begin{array}{lll}
dk(s)&=&-\bigg[A^\top(s)k(s)+\bar  A(s)^\top\mathbb E [k(s)]+Q(s)Y(s)+\bar Q(s)\mathbb E[Y(s)]\bigg]ds
\\&&-\bigg[C^\top(s)k(s)+\bar  C(s)^\top\mathbb E [k(s)]+2N_1(s)Z(s)+\bar N_1(s)\mathbb E[Z(s)]\bigg]dW(s)
\\&&-\displaystyle\int_ E \bigg[D^\top(s,e)k(s)+\bar  D(s,e)^\top\mathbb E [k(s)]+N_2(s,e)R(s,e)+\bar N_2(s,e)\mathbb E[R(s,e)]\bigg]\tilde{\mu}(d\theta, ds), s\in [t, T],
 \\k(t)&=&-GY(t)-\bar G\mathbb E[Y(t)].
\end{array}
  \right.
  \end {equation}

\end{thm}
\begin{proof}
 Let $u(\cdot)\in {\cal A}[t, T]$
 be a given  admissible  control.
 Then for any admissible control $v(\cdot) \in {{\cal A}[t, T]},$ from Lemma \ref{lem:2.3}, we have

\begin{eqnarray}\label{eq:3.3}
\begin{split}
 &\langle J'(t,\xi; u(\cdot)),  v(\cdot) \rangle
 \\=&2\mathbb
E\bigg[\int_t^T\bigg(\langle QY^{(t,\xi,u)},  Y^{(t,0,v)}\rangle
 +\langle \bar Q\mathbb E [Y^{(x,u)}],  \mathbb E[Y^{(t,0,v)}]\rangle
 +\langle N_1Z^{(t,\xi,u)},  Z^{(t,0,v)}\rangle
 \\&+\langle \bar N_1\mathbb E [Z^{(t,\xi,u)}],  \mathbb E[Z^{(t,0,v)}]\rangle
  +\int_{E}\langle N_2R^{(t,\xi,u)},  R^{(t,0,v)}\rangle \nu(de)
\\&+\int_{E}\langle \bar N_2\mathbb E [R^{(t,\xi,u)}],  \mathbb E[R^{(t,0,\xi)}]\rangle \nu(de)
+\langle N_3u, v\rangle+\langle \bar N_3\mathbb E[u], \mathbb E [v]\rangle \bigg)ds\bigg]
\\&+2\mathbb E\bigg[\langle G Y^{(t,\xi,u)}(t),  Y^{(t,0,v)}(t)\rangle+\langle
\bar G\mathbb E[Y^{(t,\xi,u)}(t)],  \mathbb E[Y^{(t,0,v)}(t)]\rangle\bigg].
\end{split}
\end{eqnarray}
On the other hand, by [23], we know that
\eqref{eq:3.2}
 admits a unique adapted solution $k(\cdot).$ Applying It\^{o}s formula to
$\langle Y^{t,0; u}(s), k(s)\rangle$
 and taking expectation, we
have
\begin{eqnarray}\label{eq:3.4}
\begin{split}
&\mathbb
E\bigg[\int_t^T\bigg(\langle QY^{(t,\xi,u)},  Y^{(t,0,v)}\rangle
 +\langle \bar Q\mathbb E [Y^{(x,u)}],  \mathbb E[Y^{(t,0,v)}]\rangle
 +\langle N_1Z^{(t,\xi,u)},  Z^{(t,0,v)}\rangle
+\langle \bar N_1\mathbb E [Z^{(t,\xi,u)}],  \mathbb E[Z^{(t,0,v)}]\rangle
  \\&+\int_{E}\langle N_2R^{(t,\xi,u)},  R^{(t,0,v)}\rangle \nu(de)+\int_{E}\langle \bar N_2\mathbb E [R^{(t,\xi,u)}],  \mathbb E[R^{(t,0,\xi)}]\rangle \nu(de) \bigg)ds\bigg]
  \\&+\mathbb E\bigg[\langle G Y^{(t,\xi,u)}(t),  Y^{(t,0,v)}(t)\rangle+\langle
\bar G\mathbb E[Y^{(t,\xi,u)}(t)],  \mathbb E[Y^{(t,0,v)}(t)]\rangle\bigg]
\\ =&\mathbb
E\bigg [\int_t^T
\bigg(\langle k, Bv+\bar B\mathbb E [v]\rangle\bigg)ds\bigg ]
\\=& \mathbb
E\bigg[\int_t^T \bigg \langle B^\top k+\bar B^\top \mathbb E [k], v\bigg\rangle ds\bigg].
\end{split}
\end{eqnarray}
Putting \eqref{eq:3.4} into \eqref{eq:3.3}, we get
\begin{eqnarray}\label{eq:4.5}
\begin{split}
 &\mathbb
\langle J'(t,\xi;u(\cdot)),  v(\cdot) \rangle=
2\mathbb E\bigg[\int_t^T \bigg \langle N_3(s)u(s)+\bar N_3(s)\mathbb E[u(s)] +B^\top(s)k({s-})+\bar B^\top(s) \mathbb E [k({s-})], v(s)\bigg\rangle ds\bigg].
\end{split}
\end{eqnarray}
For the necessary,  let $(u(\cdot); Y(\cdot), Z(\cdot), R(\cdot,\cdot))$ be an optimal pair,  then from Theorem  \ref{thm:2.5},  we have $ \langle J'( t, \xi; u(\cdot)),  v(\cdot) \rangle=0$
which imply that
\begin{eqnarray} \label{eq:3.60}
  \begin{split}
 &N_3(s)u(s)+\bar N_3(s)\mathbb E[u(s)]+B^\top(s)k({s-})+\bar B^\top(s) \mathbb E [k({s-})]=0, \quad a.e. a.s., s\in [t, T],
     \end{split}
\end{eqnarray} from \eqref{eq:4.5},  since
$v(\cdot)$ is arbitrary.

For the sufficient part, let $(u(\cdot); Y(\cdot), Z(\cdot), R(\cdot,\cdot))$ be an admissible pair satisfying \eqref{eq:3.1000}.
Putting \eqref{eq:3.1000} into \eqref{eq:4.5},
then we have  $ \langle J'(t,\xi; u(\cdot)),  v(\cdot) \rangle=0,$
which implies that $(u(\cdot); Y(\cdot), Z(\cdot), R(\cdot, \cdot))$ is an optimal control
pair from Theorem \ref{thm:2.5}.
\end{proof}
\vspace{1mm}

Finally, we introduce the so-called stochastic Hamilton system which
consists of the state equation \eqref{eq:1.1}, the adjoint equation
\eqref{eq:3.2}  and the dual
representation \eqref{eq:3.1000}:

\begin{numcases}{}\label{eq:3.7}
  dY(s)=\bigg\{A(s)Y(s)+\bar A(s)\mathbb E [Y(s)]
  +B(s)u(s)+\bar B(s)\mathbb E [u(s)]+C(s)Z(s)+\bar C(s)\mathbb E [Z(s)]\nonumber
  \\\quad\quad\quad\quad-\displaystyle\int_{E} D(s,e)R(s,e)\nu (de)
  +\displaystyle\int_{E}\bar D(s,e) \mathbb E [R(s,e)]\nu (de)\bigg\}ds+ Z(s)dW(s)+\displaystyle
  \int_{E} R(t,e)\tilde{\mu}(d\theta,
  ds), \nonumber\\
dk(s)=-\bigg[A(s)^\top k(s)+\bar  A(s)^\top\mathbb E [k(s)]+Q(s)Y(s)+\bar Q(s)\mathbb E[Y(s)]\bigg]ds\nonumber
\\\quad\quad\quad\quad
-\bigg[C(s)^\top k(s)+\bar  C(s)^\top\mathbb E [k(s)]+N_1(s)Z(s)+\bar N_1(s)\mathbb E[Z(s)]\bigg]dW(s)\nonumber
\\\quad\quad\quad\quad-\displaystyle\int_ E \bigg[D(s,e)^\top k(s)+\bar  D(s,e)^\top\mathbb E [k(s)]+N_2(s,e)R(s,e)+\bar N_2(s,e)\mathbb E[R(s,e)]\bigg]\tilde{\mu}(d\theta, ds),
 \\ Y(T)=\xi, k(t)=-GY(t)-\bar G\mathbb E[Y(t)], \nonumber
 \\ N_3(s)u(s)+\bar N_3(s)\mathbb E[u(s)]+B^\top(s)k({t-})+\bar B^\top(s) \mathbb E [k({t-})]=0, \quad s\in [t, T]. \nonumber
\end{numcases}
This is a fully coupled mean-field forward- backward stochastic differential equation
 (MF-FBSDE in short) with jump and its solution consists of $( u(\cdot), Y(\cdot), Z(\cdot),R(\cdot, \cdot), K(\cdot))$.

\begin{thm} \label{thm:6.7}
Let Assumptions \ref{ass:1.1} and \ref{ass:1.2}
be satisfied. Then
stochastic Hamilton system \eqref{eq:3.7} has a unique solution
$(u(\cdot), Y(\cdot),Z(\cdot),
 R(\cdot, \cdot), k(\cdot))\in M_{\mathscr{F}}^2(t,
T;\mathbb R^m)\times S_{\mathscr{F}}^2(t,
T;\mathbb R^n)\times M_{\mathscr{F}}^2(t,
T;\mathbb R^n)\times
{M}_{\mathscr{F}}^{\nu,2}{([0,T]\times  Z;
\mathbb R^n)} \times S_{\mathscr{F}}^2(t, T;\mathbb R^n).$  Moreover $u(\cdot)$ is the unique optimal
control of Problem \ref{pro:1.1} and $(Y(\cdot),Z(\cdot),
 R(\cdot, \cdot))$ is its
corresponding optimal state process.
\end{thm}
\begin{proof}  By Theorem \ref{them:b1},
  Problem \ref{pro:1.1} admits a unique optimal
pair $   (u(\cdot), Y(\cdot),Z(\cdot),
 R(\cdot, \cdot), k(\cdot)).$  Suppose   $k(\cdot)$
is the unique
solution of the adjoint equation \eqref{eq:3.2} corresponding to  the optimal pair
$(u(\cdot), Y(\cdot),Z(\cdot),
 R(\cdot, \cdot))$.
Then by  the necessary part of Theorem \ref{thm:b2}, the optimal control has the dual presentation \eqref{eq:3.1000}. Consequently,
 $(u(\cdot), Y(\cdot),Z(\cdot),
 R(\cdot, \cdot), k(\cdot))$ consists of an adapted
solution to the  stochastic Hamilton system \eqref{eq:3.7}.
Next, if the stochastic Hamilton system \eqref{eq:3.7} has an
another adapted solution $(\bar u(\cdot),\bar Y(\cdot), \bar Z(\cdot),\bar R(\cdot,\cdot), \bar k(\cdot)),$  then $(\bar u(\cdot),\bar Y(\cdot), \bar Z(\cdot),\bar R(\cdot,\cdot))$
must be an optimal pair of Problem \ref{pro:1.1} by the sufficient
part of Theorem \ref{thm:b2}. So we must have $  u(\cdot)= \bar
u(\cdot)$ by the  uniqueness of  the optimal control. Furthermore,
by the uniqueness of the solutions  of MF-SDE and  MF-BSDE, one
must have $(\bar u(\cdot),\bar Y(\cdot), \bar Z(\cdot),\bar R(\cdot,\cdot), \bar k(\cdot))=(u(\cdot),Y(\cdot), Z(\cdot), R(\cdot,\cdot),  k(\cdot))$ .
The proof is complete.
\end{proof}

In summary, the stochastic Hamilton system \eqref{eq:3.7}
completely characterizes the optimal control of Problem \ref{pro:1.1}. Therefore, solving Problem \ref{pro:1.1} is equivalent to solving
the stochastic Hamilton system, moreover, the unique optimal control
can be given by \eqref{eq:3.1000}.
Taking expectation to \eqref{eq:3.1000},  we have

\begin{eqnarray} \label{eq:3.8}
  \begin{split}
N_3(t)u(t)+\bar N_3(t)\mathbb E[u(t)]+B^\top(t)k({t-})+\bar B^\top(t) \mathbb E [k({t-})]=0, \quad a.e. a.s.,
     \end{split}
\end{eqnarray}
which implies that

\begin{eqnarray} \label{eq:3.9}
  \begin{split}
 &\mathbb E[u(t)]
 =-(N_3(t)+\bar N_3(t))^{-1}\bigg[(B (t)+\bar B (t))^\top \mathbb E [k({t-})]\bigg], \quad a.e. a.s.
     \end{split}
\end{eqnarray}
From\eqref{eq:3.1000}, we know that
\begin{eqnarray} \label{eq:3.10}
  \begin{split}
 &N_3(s)u(s)=-\bar N_3(s)\mathbb E[u(s)]-B^\top(s)k({s-})-\bar B^\top(s) \mathbb E [k({s-})], \quad a.e. a.s..
     \end{split}
\end{eqnarray}
Then putting \eqref{eq:3.9} into \eqref{eq:3.10},
we  have
\begin{eqnarray} \label{eq:3.6}
  \begin{split}
 u(s)=&-N_3^{-1}(s)\bigg\{B^\top (s)k({s-})+\bar B^\top (s) \mathbb E [k({s-})]\bigg]
\\&+\bar N_3(s)(N_3(s)+\bar N_3(s))^{-1}\bigg[(B (s)+\bar B (s))^\top \mathbb E [k({s-})]\bigg\}
, \quad a.e. a.s., \quad \in [t, T].
     \end{split}
\end{eqnarray}

\subsection{   Derivation of   Riccati equations }

Although the optimal control of
 Problem \ref{pro:1.1} is completely characterized by the  stochastic Hamilton system \eqref{eq:3.7},
 \eqref{eq:3.7} is a fully coupled  mean-field forward-backward stochastic differential equation whose solvability is much difficult to be obtained.
To solve  the  stochastic Hamilton system \eqref{eq:3.7}, as in [11], we will use the decoupling technique for general FBSDEs which will lead to a derivation
of two Riccati equations

Let $(u(\cdot), Y(\cdot),Z(\cdot),
 R(\cdot, \cdot), k(\cdot))$ be the solution of the stochastic  Hamilton system \eqref{eq:3.7}.Taking expectation on both sides of  the stochastic  Hamilton system \eqref{eq:3.7},
we get that $(\mathbb E[u(\cdot)], \mathbb E[Y(\cdot)],\mathbb E[Z(\cdot)],
 \mathbb E[R(\cdot, \cdot)], \mathbb E[k(\cdot)])$ satisfies  the following forward-backward
ordinary differential equation 
\begin{equation}\label{eq:4.1}
\left\{\begin {array}{lll}
dE[k(s)]&=&-\bigg[(A^\top(s)+\bar  A(s)^\top)\mathbb E [k(s)]+(Q(s)+\bar Q(s))\mathbb E[Y(s)]\bigg],\\
   d\mathbb E[Y(s)]&=&\bigg\{(A(s)+\bar A(s))\mathbb E [Y(s)]
  +(B(s)+\bar B(s))\mathbb E [u(s)]+(C(s)+\bar C(s))\mathbb E [Z(s)]
  \\&&+\displaystyle\int_{E} (D(t,e)
  +\bar D(s,e)) \mathbb E [R(s,e)]\nu (de)\bigg\}ds,
   \\\mathbb E[Y(T)]&=& \mathbb E[\xi],
   \mathbb E[k(t)]=-(G+\bar G)\mathbb E[Y(t)],\\
  0&=&(N_3(s)+\bar N_3(s))\mathbb E[u(s)]+(B^\top(s)+\bar B^\top(s)) \mathbb E [k({s-})].
\end {array}
\right.
\end{equation}
Further,  it is easy to check that  $(u(\cdot)-\mathbb E [u(\cdot)], Y(\cdot)-\mathbb E [Y(\cdot)],Z(\cdot)-\mathbb E [Z(\cdot)], R(\cdot,\cdot)-\mathbb E [R(\cdot,\cdot)],k(\cdot)-\mathbb E [k(\cdot)])$ satisfies  the following forward-backward
stochastic differential equation
\begin{equation}\label{eq:3.13}
\left\{\begin {array}{lll}
  dY-\mathbb E[Y]&=&\bigg\{A(Y-\mathbb E [Y])
  +B(u-\mathbb E [u])+C(Z-\mathbb E [Z])
  +\displaystyle\int_{E} D(R-\mathbb E [R])\nu (de)\bigg\}ds+ ZdW+\displaystyle
  \int_{E} R\tilde{\mu}(d\theta,
  ds),
   \\
dk-\mathbb E[k]&=&-\bigg[A^\top(k-\mathbb E [k])+Q(Y-\mathbb E[Y])\bigg]ds+
\bigg[C^\top(k-\mathbb E [k])-
(C^\top+\bar  C^\top)\mathbb E [k]+N_1(Z-\mathbb E[Z])\\&&+(N_1+\bar N_1)\mathbb E[Z]\bigg]dW
-\displaystyle\int_ E \bigg[D^\top(k-\mathbb E [k])+
(D^\top+\bar  D^\top)\mathbb E [k]
+N_2(R-
\mathbb E[R])\\&&+(N_2+\bar N_2)\mathbb E[R]\bigg]\tilde{\mu}(d\theta, ds),
 \\k(t)-\mathbb E[k(t)]&=&-G(Y(t)-\mathbb E[Y(t)]),Y(T)-\mathbb E[Y(T)]=\xi-\mathbb E[\xi],
 \\ &0=&N_3(u-\mathbb E[u])+B^\top(k({t-})- \mathbb E [k({t-})]).
 \end {array}
\right.
\end{equation}
Now we assume that the state process $Y(\cdot)$ and the adjoint process
  $k(\cdot)$  have the following
  relationship:
\begin{eqnarray}
  \begin{split}
    Y(s)=P(s)(k(s)-\mathbb E[k(s)])
    +\Pi(s)\mathbb  E[k(s)]+\varphi(s),
  \end{split}
\end{eqnarray}
where $P(\cdot), \Pi(\cdot):[0,T]\longrightarrow \mathbb R^{n\times n}$
are absolutely continuous and  $\varphi(\cdot)$
satisfies the following BSDE with jump
\begin{eqnarray}
  \begin{split}
    d\varphi(s)=\alpha(s)ds+\beta(s)dW(s)
    +\Phi(t,e)d\tilde \mu(de,ds),
    \varphi(T)=\xi,
  \end{split}
\end{eqnarray}
for some $\mathbb F-$progressively measurable 
process $\alpha$, $\beta$ and $\Phi$.
Consequently,  we further  have the following
relationship
\begin{eqnarray} \label{eq:4.50}
  \begin{split}
    \mathbb E [Y(s)]=\Pi(s)\mathbb  E[k(s)]
    +\mathbb E[\varphi(s)]
  \end{split}
\end{eqnarray}
 and

 \begin{eqnarray} \label{eq:4.6}
  \begin{split}
   Y(s)-\mathbb E [Y(s)]=P(s)(k(s)-\mathbb E[k(s)])+(\varphi(s)-\mathbb E[\varphi(s)]).
  \end{split}
\end{eqnarray}
Denote
\begin{eqnarray}
  d\eta(s)=\gamma ds+\beta(s)dW(s)
  +\Phi(s,e)d\tilde \mu(de,ds),
    \eta(T)=\xi-\mathbb E[\xi],
\end{eqnarray}
where $$\eta(s)= \varphi (s)-
\mathbb E[\varphi(s)], \gamma(s)= \alpha(s)-
\mathbb E[\alpha(s)].$$
In the following,  we
begin to formerly derive the
corresponding Riccati equations
which $P(\cdot)$ and $\Pi(\cdot)$
should satisfy. From the relationships
\eqref{eq:4.6} and \eqref{eq:3.13},
applying It\^{o} formula
 to $P(s)(Y(s)-\mathbb E[Y(s)])$ leads to
\begin{eqnarray} \label{eq:4.7}
  \begin{split}
&\bigg\{A(Y-\mathbb E [Y])
  +B(u-\mathbb E [u])+C(Z-\mathbb E [Z])
  +\displaystyle\int_{E} D(R-\mathbb E [R])\nu (de)\bigg\}ds+ ZdW+\displaystyle
  \int_{E} R\tilde{\mu}(d\theta,
  ds)
\\=&d\big(Y-\mathbb E[Y])
\\=&dP(k-\mathbb E[k])+d\eta
\\=& \bigg[\dot{P}(k-\mathbb E[k])-PA^\top(k-\mathbb E [k])-PQ(Y-\mathbb E[Y])\bigg]ds-P\bigg[C^\top(k-\mathbb E [k])+
(C^\top+\bar  C^\top)\mathbb E [k]\\&+N_1(Z-\mathbb E[Z])+(N_1+\bar N_1)\mathbb E[Z]\bigg]dW
-\displaystyle\int_ E  P \bigg[D^\top(k-\mathbb E [k])+
(D^\top+\bar  D^\top)\mathbb E [k]
\\&+N_2(R-
\mathbb E[R])+(N_2+\bar N_2)\mathbb E[R]\bigg]\tilde{\mu}(d\theta, ds)+\gamma ds+\beta dW
  +\Phi d\tilde \mu(de,ds).
  \end{split}
\end{eqnarray}
Comparing the diffusion terms of
both sides of the above equality, we have
\begin{eqnarray} \label{eq:4.8}
  \begin{split}
    Z =-P\bigg[C^\top k +\bar  C ^\top\mathbb E [k ]+N_1 Z +\bar N_1 \mathbb E[Z ]\bigg]+\beta ,
  \end{split}
\end{eqnarray}
\begin{eqnarray}\label{eq:4.9}
  \begin{split}
    R=-P \bigg[D^\top k +\bar  D^\top\mathbb E [k ]+N_2R+\bar N_2\mathbb E[R]\bigg]
    +\Phi,
  \end{split}
\end{eqnarray}
and
\begin{eqnarray} \label{eq:4.7}
  \begin{split}
&A P (k -\mathbb E[k ])+A \eta
  -B N_3^{-1} B^\top (k -\mathbb E [k ])+C (Z -\mathbb E [Z ])
  +\displaystyle\int_{E} D(R-\mathbb E [R])\nu (de)
  \\&=\dot{P}(k -\mathbb E[k ])-PA^\top (k -\mathbb E [k ])-P Q P (k -\mathbb E[k ])-PQ \eta +\gamma,
    \end{split}
\end{eqnarray}
which imply that
\begin{eqnarray} \label{eq:4.7}
  \begin{split}
&\bigg(\dot{P}-PA^\top-AP -PQP+BN_3^{-1}B^\top
\bigg)(k-\mathbb E[k])
  \\&-C(Z-\mathbb E [Z])
  -\displaystyle\int_{E} D(R-\mathbb E [R])\nu (de)
  -(A+PQ)\eta(t)+\gamma=0.
    \end{split}
\end{eqnarray}

Then  taking expectation on both sides of
\eqref{eq:4.8} and \eqref{eq:4.9},
we have the following relationships:

\begin{eqnarray} \label{eq:4.10}
  \begin{split}
  0=-(PN_1 +P\bar N_1 +I)\mathbb E[Z ]-P(C^\top +\bar  C ^\top)\mathbb E [k ]+\mathbb E[\beta ],
  \end{split}
\end{eqnarray}
\begin{eqnarray} \label{eq:4.11}
  \begin{split}
   0=-(PN_2+\bar PN_2+I)\mathbb E[R-P(D^\top+\bar  D^\top)\mathbb E [k ]+\mathbb E[\Phi],
  \end{split}
\end{eqnarray}
\begin{eqnarray} \label{eq:4.10}
  \begin{split}
  0=-(PN_1(t)+I)
  (Z-\mathbb E[Z])-PC^\top
  (k-\mathbb E [k])+(\beta-\mathbb E[\beta]),
  \end{split}
\end{eqnarray}
\begin{eqnarray} \label{eq:4.11}
  \begin{split}
   0=-(PN_2+I)\mathbb ( R-E[R])-PD^\top(k-\mathbb E [k])+  \Phi-\mathbb E[\Phi].
  \end{split}
\end{eqnarray}
Assume that $ PN_1+P\bar N_1+I, PN_2+\bar PN_2+I PN_2+I, PN_1+I$ are invertible, we get

\begin{eqnarray} \label{eq:3.28}
  \begin{split}
  \mathbb E[Z]=-(PN_1+P\bar N_1+I)^{-1} \bigg\{P(C^\top+\bar  C^\top)\mathbb E [k]-\mathbb E[\beta]\bigg\},
  \end{split}
\end{eqnarray}
\begin{eqnarray} \label{eq:3.29}
  \begin{split}
   \mathbb E[R]=-(PN_2+\bar PN_2 +I)^{-1}\bigg\{P(D^\top +\bar  D ^\top)\mathbb E [k]-\mathbb E[\Phi ]\bigg\},
  \end{split}
\end{eqnarray}
\begin{eqnarray} \label{eq:3.30}
  \begin{split}
  (Z -\mathbb E[Z ])=-(PN_1 +I)^{-1}\bigg\{PC^\top )
  (k -\mathbb E [k ])-(\beta -\mathbb E[\beta ])\bigg\},
  \end{split}
\end{eqnarray}
\begin{eqnarray} \label{eq:3.31}
  \begin{split}
   \mathbb ( R -E[R ])=-(PN_2 +I)^{-1}
   \bigg\{PD^\top (k -\mathbb E [k ]-(\Phi -\mathbb E[\Phi ])\bigg\}.
  \end{split}
\end{eqnarray}
Putting \eqref{eq:3.30} and 
\eqref{eq:3.31} into \eqref{eq:4.7}, we get
 that
\begin{eqnarray} \label{eq:3.32}
  \begin{split}
&\bigg(\dot{P}-PA^\top-AP -PQP+BN_3^{-1}B^\top
+ C(PN_1 +I)^{-1}PC^\top 
+\int_ED(PN_2 +I)^{-1}PD^\top 
\nu(de)\bigg)(k-\mathbb E[k])
  \\&-C(PN_1 +I)^{-1}(\beta -\mathbb E[\beta ])
  -\displaystyle\int_{E} D(PN_2 +I)^{-1}
   (\Phi -\mathbb E[\Phi ])\nu (de)
  -(A+PQ)\eta(t)+\gamma=0.
    \end{split}
\end{eqnarray}
from which one should let
\begin {equation}\label{eq:4.18}
\left\{\begin{array}{lll}
 &\dot{P}-PA^\top-AP -PQP+BN_3^{-1}B^\top
+ C(PN_1 +I)^{-1}PC^\top
+\int_ED(PN_2 +I)^{-1}PD^\top
\nu(de)
=0,
\\&  \gamma-C(PN_1 +I)^{-1}(\beta -\mathbb E[\beta ])
  -\displaystyle\int_{E} D(PN_2 +I)^{-1}
   (\Phi -\mathbb E[\Phi ])\nu (de)
  -(A+PQ)\eta =0.
 \end{array}
  \right.
  \end {equation}

Furthermore,  from
\eqref{eq:4.50} and \eqref{eq:4.1}, we have

\begin{eqnarray}
  \begin{split}
  &\bigg\{(A +\bar A )\mathbb E [Y ]
  -(B +\bar B (N +\bar N )^{-1}(B  +\bar B )^\top \mathbb E [k]+(C +\bar C )\mathbb E [Z ]
  +\displaystyle\int_{E} (D
  +\bar D) \mathbb E [R]\nu (de)\bigg\}ds
   \\&=d\mathbb E [Y ]
   \\&=d\Pi \mathbb  E[k ]
    +d\mathbb E[\varphi ]
   \\&=\bigg[\dot{\Pi}\mathbb E[k ]-\Pi(A^\top +\bar  A ^\top)\mathbb E [k ]-\Pi(Q +\bar Q )\mathbb E[Y ]+\mathbb E[\alpha ]\bigg]ds.
  \end{split}
\end{eqnarray}
Putting \eqref{eq:4.50}, \eqref{eq:3.28}
 and \eqref{eq:3.29} into the left hand
 of the above equality
 and comparing both sides  of the above equality, we get

 \begin{eqnarray} \label{eq:4.23}
  \begin{split}
  0&=\bigg\{\dot{\Pi}
  -\Pi(A^\top +\bar  A ^\top)-(A +\bar A )\Pi -\Pi(Q +\bar Q )\Pi 
+(B +\bar B )(N_3 +\bar N_3 )^{-1}(B +\bar B  )^\top
  \\&+(C +\bar C )(PN_1+P\bar N_1+I)^{-1} P(C^\top +\bar  C ^\top)+\displaystyle\int_{E} (D+\bar D ) (PN_2+\bar PN_2 +I)^{-1}P(D^\top +\bar  D ^\top)\bigg\}k
    \\&+\bigg [-\Pi(Q +\bar Q )-(A +\bar A )\bigg ]\mathbb E[\varphi ]+\mathbb E[\alpha ]-(C +\bar C )(PN_1 +P\bar N_1 +I)^{-1} \mathbb E[\beta ]
  \\&-\displaystyle\int_{E} (D
  +\bar D) (PN_2+P\bar N_2+I)^{-1}\mathbb E[\Phi]\nu (de).
  \end{split}
\end{eqnarray}
Hence, one should let
\begin {equation}\label{eq:3.36}
\left\{\begin{array}{lll}
 &\dot{\Pi}
  -\Pi(A^\top +\bar  A ^\top)-(A +\bar A )\Pi -\Pi(Q +\bar Q )\Pi
+(B +\bar B )(N_3 +\bar N_3 )^{-1}(B +\bar B  )^\top
  \\&+(C +\bar C )(PN_1+P\bar N_1+I)^{-1} P(C^\top +\bar  C ^\top)+\displaystyle\int_{E} (D+\bar D ) (PN_2+\bar PN_2 +I)^{-1}P(D^\top +\bar  D ^\top)=0,
 \\& \mathbb E[\alpha ]-\big [\Pi(Q +\bar Q )+(A +\bar A )\big]\mathbb E[\varphi ]-(C +\bar C )(PN_1 +P\bar N_1 +I)^{-1} \mathbb E[\beta ]
  -\displaystyle\int_{E} (D
  +\bar D) (PN_2+P\bar N_2+I)^{-1}\mathbb E[\Phi]\nu (de).
\end{array}
  \right.
  \end {equation}
  Moreover, comparing the terminal values on both sides of the two equations in \eqref{eq:4.50} and \eqref{eq:4.6}, one has   $$\Pi(T)=0, P(T)=0.$$
  Therefore, by \eqref{eq:4.18} and \eqref{eq:3.36}, we see that $P(\cdot)$ and $\Pi(\cdot)$ should satisfy the following Riccati-type
equations, respectively:

\begin {equation}\label{eq:4.37}
\left\{\begin{array}{lll}
 &\dot{P}-PA^\top-AP -PQP+BN_3^{-1}B^\top
+ C(PN_1 +I)^{-1}PC^\top
+\int_ED(PN_2 +I)^{-1}PD^\top
\nu(de)
=0,
\\& P(T)=0,
 \end{array}
  \right.
  \end {equation}
and 
\begin {equation}\label{eq:4.38}
\left\{\begin{array}{lll}
 &\dot{\Pi}
  -\Pi(A^\top +\bar  A ^\top)-(A +\bar A )\Pi -\Pi(Q +\bar Q )\Pi
+(B +\bar B )(N_3 +\bar N_3 )^{-1}(B +\bar B  )^\top
  \\&+(C +\bar C )(PN_1+P\bar N_1+I)^{-1} P(C^\top +\bar  C ^\top)+\displaystyle\int_{E} (D+\bar D ) (PN_2+\bar PN_2 +I)^{-1}P(D^\top +\bar  D ^\top)=0,
  \\&\Pi(T)=0,
\end{array}
  \right.
  \end {equation}
and $\varphi(t)$ should satisfy the following MF-BSDE on $[0,T]$:
\begin {equation}\label{eq:4.39}
\left\{\begin{array}{lll}
d\varphi=&\bigg\{(PQ+A)\varphi
+C(PN_1+I)^{-1}\beta+\displaystyle\int_{E} D(PN_2+I)^{-1} \Phi\nu (de)
\\&+(\bar A+\Pi(Q+\bar Q)-PQ)\mathbb E[\varphi]+\big[(C+\bar C)(PN_1+P\bar N_1+I)^{-1}-C(PN_1+I)^{-1}\big]\mathbb E[\beta]
\\&+\displaystyle\int_{E} \big[ (D
  +\bar D) (PN_2+\bar PN_2+I)^{-1}-D(PN_2+I)^{-1}\big]\mathbb E[\Phi]\nu (de)\bigg\}ds
\\&+\beta dW
    +\Phi d\tilde \mu(de,ds),
    \\
\varphi(T)=& \xi.
\end{array}
  \right.
  \end {equation}

By the same argument as in section 4 of [11], under Assumptions
 \ref{ass:1.1} and \ref{ass:1.2}, we can get that
Riccati equations \eqref{eq:4.37} and
\eqref{eq:4.38} have a unique solution,
respectively.

\subsection{Representations of optimal controls}
This section is going to give explicit formulas of the optimal controls and the value function,
via the solutions to the Riccati equations \eqref{eq:4.37}, \eqref{eq:4.38}, and the MF-BSDE \eqref{eq:4.39}. Now
 we state our main result as follows.
\begin{thm}
  Let Assumption\ref{ass:1.1} and \ref{ass:1.2} be satisfied.
 Let $P(\cdot)$ and $\Pi(\cdot)$ be
the unique solutions to the Riccati equations  \eqref{eq:4.37} and \eqref{eq:4.38} , respectively, and let  $(\varphi(\cdot),
\beta(\cdot), \Phi(\cdot))$
be the unique adapted solution to the MF-BSDE 
\eqref{eq:4.39}. Then the following MF-FSDE admits a unique solution $k(\cdot)$:

\begin {equation}\label{eq:3.40}
\left\{\begin{array}{lll}
dk &=&-\bigg[(A^\top +Q P)k +
(\bar  A ^\top-Q P
+(Q +\bar Q )\Pi )\mathbb E [k ]+Q \varphi+\bar Q \mathbb E[\varphi]\bigg]ds
- \bigg[\big[I-N_1 (PN_1+I)^{-1}P\big]
C^\top k\\&&+\big[\bar  C ^\top+N_1 (PN_1+I)^{-1}PC^\top-(N_1 +\bar N_1 )(PN_1+P\bar N_1+I)^{-1} P(C^\top+\bar  C^\top)\big]\mathbb E [k ]
\\&&-N_1 (PN_1+I)^{-1}
  (\beta-\mathbb E[\beta])-(N_1 +\bar N_1 )(PN_1+P\bar N_1+I)^{-1} \mathbb E[\beta]\bigg ]dW\\&&-\displaystyle\int_ E \bigg[\big[I-N_2(PN_2+I)^{-1}P\big]
D^\top k
+\big[\bar  D ^\top+2N_2 (PN_2+I)^{-1}PD^\top
\\&&-(N_2 +\bar N_2 )(PN_2+P\bar N_2+I)^{-1} P(D^\top+\bar  D^\top)\big]\mathbb E [k ]\\&&-N_2(PN_2+I)^{-1}
  (\Phi-\mathbb E[\Phi])-(N_2+\bar N_2)(PN_2+P\bar N_2+I)^{-1} \mathbb E[\Phi]\bigg\}\tilde{\mu}(d\theta, ds), s\in [t, T],
 \\k(t)&=&-(I+GP)^{-1}G(\varphi(t)-\mathbb E[\varphi(t)])-(I+(G+\bar G)\Pi)^{-1}(G+\bar G)\mathbb E[\varphi(t)],
\end{array}
  \right.
  \end {equation}
and the unique optimal control of Problem
  \ref{pro:1.1} for the terminal state $\xi$ is given by
\begin{eqnarray} \label{eq:3.41}
  \begin{split}
 u =&-N^{-1}_3 B^\top  \big[k+  \mathbb E [k]\big]- N_3^{-1}\bar N_3(N_3 +\bar N_3 )^{-1}(B+\bar B )^\top \mathbb E [k]
, \quad a.e. a.s.
     \end{split}
\end{eqnarray}

\end{thm}

\begin{proof}
Let $P(\cdot)$ and $\Pi(\cdot)$ be
the unique solutions to the Riccati equations  \eqref{eq:4.37} and \eqref{eq:4.38} , respectively, and let  $(\varphi(\cdot),
\beta(\cdot), \Phi(\cdot))$
be the unique adapted solution to the MF-BSDE
\eqref{eq:4.39}. It is clear that \eqref{eq:3.40} has a unique solution $k(\cdot)$. So we need only prove that
$u(\cdot)$
defined by \eqref{eq:3.41} is the unique optimal control of Problem \ref{pro:1.1} for the terminal state $\xi$. To this end, 
define
\begin{eqnarray}
  \begin{split}
    Y=:P(k-\mathbb E[k])
    +\Pi\mathbb  E[k]+\varphi,
  \end{split}
\end{eqnarray}

\begin{eqnarray} \label{eq:3.43}
  \begin{split}
  Z =:&-(PN_1 +I)^{-1}\bigg [PC^\top )
  (k -\mathbb E [k ])-(\beta -\mathbb E[\beta ])\bigg]-(PN_1 +P\bar N_1 +I)^{-1} \bigg [P(C^\top +\bar  C ^\top)\mathbb E [k ]-\mathbb E[\beta ]\bigg],
  \end{split}
\end{eqnarray}

\begin{eqnarray} \label{eq:4.11}
  \begin{split}
   R=:&-(PN_2+I)^{-1}
   \bigg [PD^\top(k-\mathbb E [k]-  (\Phi-\mathbb E[\Phi]\bigg]-(PN_2+ P\bar N_2+I)^{-1}\bigg[P(D^\top+\bar  D^\top)\mathbb E [k]-\mathbb E[\Phi]\bigg].
  \end{split}
\end{eqnarray}
  Then following the derivation of riccati 
  equation
in the previous subsection,
by applying It\^{o} formula to
$Y(s)=P(s)(k(s)-\mathbb E[k(s)])
    +\Pi(s)\mathbb  E[k(s)]+\varphi(s),$
it is easy to check that $(u(\cdot), Y(\cdot),Z(\cdot),
 R(\cdot, \cdot), k(\cdot))$ satisfies the following stochastic Hamilton system
\begin{numcases}{}\label{eq:4.45}
  dY(s)=\bigg\{A(s)Y(s)+\bar A(s)\mathbb E [Y(s)]
  +B(s)u(s)+\bar B(s)\mathbb E [u(s)]+C(s)Z(s)+\bar C(s)\mathbb E [Z(s)]\nonumber
  \\\quad\quad\quad\quad-\displaystyle\int_{E} D(s,e)R(s,e)\nu (de)
  +\displaystyle\int_{E}\bar D(s,e) \mathbb E [R(s,e)]\nu (de)\bigg\}ds+ Z(s)dW(s)+\displaystyle
  \int_{E} R(t,e)\tilde{\mu}(d\theta,
  ds), \nonumber\\
dk(s)=-\bigg[A(s)^\top k(s)+\bar  A(s)^\top\mathbb E [k(s)]+Q(s)Y(s)+\bar Q(s)\mathbb E[Y(s)]\bigg]ds\nonumber
\\\quad\quad\quad\quad
-\bigg[C(s)^\top k(s)+\bar  C(s)^\top\mathbb E [k(s)]+N_1(s)Z(s)+\bar N_1(s)\mathbb E[Z(s)]\bigg]dW(s)\nonumber
\\\quad\quad\quad\quad-\displaystyle\int_ E \bigg[D(s,e)^\top k(s)+\bar  D(s,e)^\top\mathbb E [k(s)]+N_2(s,e)R(s,e)+\bar N_2(s,e)\mathbb E[R(s,e)]\bigg]\tilde{\mu}(d\theta, ds),
 \\ Y(T)=\xi, k(t)=-GY(t)-\bar G\mathbb E[Y(t)], \nonumber
 \\ N_3(s)u(s)+\bar N_3(s)\mathbb E[u(s)]+B^\top(s)k({s-})+\bar B^\top(s) \mathbb E [k({s-})]=0, \quad s\in [t, T]. \nonumber
\end{numcases}
Thus,  by Theorem \ref{thm:6.7}, we know that $ u(\cdot)$ defined by \eqref{eq:3.41}  is the 
unique optimal control of Problem \ref{pro:1.1}.
The proof is complete.
\end{proof}

\end{document}